\documentclass[12pt,a4paper]{amsart}
\usepackage{amsfonts}
\usepackage[top=35mm, bottom=35mm, left=30mm, right=30mm]{geometry}
\usepackage[colorlinks=true,citecolor=blue]{hyperref}
\usepackage{mathptmx}
\usepackage{eucal}
\usepackage{graphicx}
\usepackage{amssymb}
\usepackage{amsmath}
\usepackage{amsthm}
\usepackage{enumerate}
\usepackage{verbatim}

\usepackage{xcolor}
\newtheorem{thm}{Theorem}[section]
\newtheorem{cor}[thm]{Corollary}
\newtheorem{lem}[thm]{Lemma}
\newtheorem{prop}[thm]{Proposition}

\theoremstyle{definition}
\newtheorem{defn}[thm]{Definition}

\numberwithin{equation}{section}

\newcommand{\N}{\mathbb{N}}

\newcommand{\Z}{\mathbb{Z}}

\newcommand{\Orb}{\mathrm{Orb}}

\DeclareMathOperator{\supp}{supp}
\DeclareMathOperator{\diam}{diam}

\newcommand{\ep}{\epsilon}

\newcommand{\ra}{\rightarrow}

\def \B {\mathcal B}

\begin{document}

\title{Relative mean sensitivity and relative entropy under group actions}
\author[W. Ouyang]{Wenqi Ouyang}

\address[Wenqi Ouyang]{Department of Mathematics, Shantou University, Shantou 515063, P. R. China}
\email{}

\begin{abstract}
In this paper we introduce relative mean sensitivity
both in topological and measure-theoretical settings under group actions.
It turns out that relative positive entropy is relative mean $n$-sensitive between minimal systems;
relative positive $\mu$-entropy for an ergodic measure $\mu$ is relative mean $\mu$-$n$-sensitive.
\end{abstract}

\subjclass[2010]{37A35, 37B40}
\keywords{Relative entropy; relative mean sensitivity}

\maketitle
\section{Introduction}

By a $\mathbb{Z}$-action {\it topological dynamical system} ({\it TDS} for short) we mean a pair $(X,T)$,
where $X$ is a compact metric space with a metric $d$ and $T$ is a homeomorphism from $X$ to itself.
Denote by  $\B_X$ all Borel measurable subsets of $X$. A Borel (probability) measure $\mu$ on $X$ is called $T$-\textit{invariant} if  $\mu(T^{-1}A)=\mu(A)$ for any $A\in \mathcal{B}_X$.
A $T$-invariant  measure $\mu$ on $X$ is called \textit{ergodic} if  $B\in \mathcal{B}_X$
with $T^{-1}B=B$ implies $\mu(B)=0$ or $\mu(B)=1$.
Each $T$-invariant  measure induces a $\mathbb{Z}$-action
{\it measure-preserving system} ({\it MPS} for short)  $(X,\B_X,\mu, T)$.

In 1958 Kolmogorov \cite{K58} associated to any MPS an isomorphism invariant, namely
the measure-theoretical entropy. Later on Adler et al \cite{AKM65} introduced topological entropy  in any TDS.
Sensitivity
is one of the most fundamental concepts in dynamical theory.
Recall that a TDS $(X,T)$ is sensitive if there is $\delta>0$  such that
for any non-empty open subset $U$ of $X$, there are $x_1,x_2\in U$ and $m\in \Z_+$ such that $d(T^{m}x_1,T^{m}x_2)>\delta$.
In \cite{X05} Xiong introduced a strong form of sensitivity named $n$-sensitivity.
Ye and Zhang \cite{YZ08} showed that
transitive TDS with positive entropy is $n$-sensitive.
Parallelly in \cite{HLY11} Huang, Lu and Ye studied the above notions in MPS $(X,\B_X,\mu, T)$
by introducing the notions of $\mu$-$n$-sensitivity.
They showed that
 ergodic MPS with positive entropy is $\mu$-$n$-sensitive.

In \cite{LTY15} Li, Tu and Ye introduced the notion of mean sensitivity
and showed that a minimal TDS with positive entropy is mean sensitive.
Li, Ye and Yu \cite{LYY22, LY21} studied mean $n$-sensitivity and show that
 a minimal TDS with positive entropy is mean $n$-sensitive.
The parallel notion of mean sensitivity for an ergodic MPS $(X,\B_X,\mu, T)$ was studied by Garc\'{\i}a-Ramos \cite{G17}.
In \cite{L16} Li studied mean $\mu$-$n$-sensitivity
and showed that positive $\mu$-entropy ergodic MPS is
mean $\mu$-$n$-sensitive.

For a factor map between two TDSs, Ledrappier and Walters \cite{LW77}
introduced notions of relative
topological entropy and relative measure-theoretical entropy
for a measure $\mu$.
Xu and Yu \cite{XY2026} introduced relative mean sensitivity
both in topological and measure-theoretical settings.
They showed that
relative positive entropy is relative mean $n$-sensitive between minimal systems;
relative positive $\mu$-entropy for an ergodic measure $\mu$ is relative mean $\mu$-$n$-sensitive.

Comparing to dynamical systems of $\mathbb{Z}$-action, the level of development of dynamical systems
of an amenable group action lagged behind.
Ornstein and Weiss \cite{OW87} laid a foundation of an amenable group action in 1987.
Huang, Ye and Zhang \cite{HYZ07} studied entropy theory for a countable discrete amenable
group action both in topological and measure-theoretical settings.
Dooley and Zhang \cite{DZ15} studied
relative entropy theory for a countable discrete amenable
group action both in topological and measure-theoretical settings.
In this paper we will introduce relative mean sensitivity
both in topological and measure-theoretical settings under group actions.
It turns out that relative positive entropy is relative mean $n$-sensitive between minimal systems;
relative positive $\mu$-entropy for an ergodic measure $\mu$ is relative mean $\mu$-$n$-sensitive.

This paper is organized as follows. In Section \ref{sect:preliminaries} we introduce some basic notions and facts that we will use later.  In Section \ref{sect:Ergodic decompositions}, we study ergodic decompositions of  conditional entropy.
Sections \ref{sect:MS-tuple:measure} and \ref{sect:MS-tuple:topological} are devoted to the study of
relative mean sensitivity and relative entropy in measure-theoretical and topological settings respectively.

\section{Preliminaries}\label{sect:preliminaries}
%In this section we introduce some basic notions and facts that we will use later.
Throughout the paper, denote $\mathbb{N}$ and ${\mathbb{Z}}_{+}$ by the
collection of natural numbers and non-negative integers respectively.

A countable discrete
group $G$ is said to be {\it amenable} if there exists a
sequence of finite non-empty subsets $F_n\subset  G$ such that for every
$g \in G$,
$\lim_{n\rightarrow +\infty} \frac{|gF_n\Delta F_n|}{|F_n|}=0$ and $\lim_{n\rightarrow +\infty} \frac{|F_ng\Delta F_n|}{|F_n|}=0$,
and in this case the sequence $\{F_n\}_{n=1}^\infty$  is called a {\it (two-sided)
F{\o}lner sequence}.
From now on, we will fix $G$ to be a countable discrete infinite amenable group.

By a {\it $G$-system} $(X, G)$ we
mean that $X$ is a compact metric space and $\Gamma: G\times
X\rightarrow X, (g, x)\mapsto g x$ is a continuous mapping
 satisfying
\begin{enumerate}

\item $\Gamma (e_G, x)= x$ for each $x\in X$ where $e_G$ is the unit element of $G$,

\item $\Gamma (g_1, \Gamma (g_2, x))= \Gamma (g_1 g_2, x)$ for each
$g_1, g_2\in G$ and $x\in X$.
\end{enumerate}

Let $\mathcal{F}=\{F_m\}_{m=1}^\infty$ to be a F{\o}lner sequence of $G$.
For $F\subseteq G$, we define the \emph{upper density} of $F$ with respect to $\mathcal{F}$ by
$$
\overline{D}_{\mathcal{F}}(F)=\limsup_{m\to\infty} \frac{|F\cap F_m|}{|F_m|}.
$$
Similarly we define the \emph{lower density} of $F$ with respect to $\mathcal{F}$ by
$$
\overline{D}_{\mathcal{F}}(F)=\liminf_{m\to\infty} \frac{|F\cap F_m|}{|F_m|}.
$$
If $\overline{D}_{\mathcal{F}}(F)=\underline{D}_{\mathcal{F}}(F)$,
we denote the common value by $D_{\mathcal{F}}(F)$ and call it the \emph{density} of $F$
with respect to $\mathcal{F}$.

\begin{lem}\label{20260620}
Let $A\subset G$ and $g\in G$, then $\overline{D}_{\mathcal{F}}(A)=\overline{D}_{\mathcal{F}}(Ag)$.
\end{lem}
\begin{proof}
$\overline{D}_{\mathcal{F}}(Ag)=\limsup_{m\to\infty} \frac{|Ag\cap F_m|}{|F_m|}$

$=\limsup_{m\to\infty} \frac{|Ag\cap F_m|}{|F_m|}+\lim_{m\rightarrow +\infty} \frac{|F_mg\Delta F_m|}{|F_m|}$

$=\limsup_{m\to\infty} \frac{|Ag\cap F_mg|}{|F_m|}=\limsup_{m\to\infty} \frac{|A\cap F_m|}{|F_m|}=\overline{D}_{\mathcal{F}}(A)$
\end{proof}

\medskip
Let $(X, G)$ be a TDS and $x\in X$. We say that $\textrm{Orb}(x,G)=\{g x\colon g\in G\}$ is the \textit{orbit} of $x$.
Fix $n\in \N$, we denote by $X^{(n)}$  the $n$-fold product of $X$, by $\Delta_n(X)=\{(x,x,\dots, x)\in X^{(n)}\colon x\in X\}$  the diagonal of $ X^{(n)}$
and  $\Delta_n^\prime(X)=\{(x_1,x_2,\dots,x_n)\in X^{(n)}:  x_i=x_j  \text{ for some } 1\le i\neq j\le n \}$.

If a closed subset $Y\subseteq X$ is $G$-invariant in the sense of $GY\subseteq Y$,
then the restriction $(Y, G)$  is also a TDS, which is referred to be a \textit{subsystem} of $(X,G)$.

Let $x\in X$ and $U,V\subseteq X$. Define
$$
N(x,U)=\{g\in G \colon g x\in U\} \ \text{ and }\  N(U,V)=\{g\in G: U\cap g^{-1}V\neq\emptyset\}.
$$
A TDS $(X,G)$ is called \textit{transitive} if $N(U,V)\neq\emptyset$ for any  non-empty open subsets $U,V$ of $X$.

A point $x\in X$ is called a \textit{transitive point} if $\overline{\Orb(x,G)}=X$. It is known that a transitive TDS has a dense $G_\delta$ set of transitive points.  A TDS $(X,G)$ is called \textit{minimal} if all points in $X$ are transitive.

Let $(X, G)$ and $(Y,G)$ be two TDSs.
A  map $\pi\colon X\to Y$  is called a \textit{factor map}
if $\pi$ is continuous onto and $\pi\circ g=g\circ\pi$ for any $g\in G$,
and in which case $(Y,G)$ is often referred to be a \textit{factor} of $(X, G)$.
Put $R_{n,\pi} = \{(x_1,\dots,x_n) \in X^{(n)} : \pi(x_i)=\pi(x_j),\ \forall \ 1\le i, j \le n\}$ for $n\geq 2$.

\medskip
Let $(X,G)$ be a TDS. A (probability) measure $\mu$ on $X$ is called $G$-\textit{invariant}
if  $\mu(gA)=\mu(A)$ for any $A\in \mathcal{B}_X$ and any $g\in G$,
where $\mathcal{B}_X$ is the collection of all Borel sets in $X$.
By Krylov-Bogolioubov theorem there exists a $G$-invariant measure $\mu$ on $X$.
A $G$-invariant  measure $\mu$ on $X$ is called \textit{ergodic} if  $B\in \mathcal{B}_X$
with $gB=B$ for any $g\in G$ implies $\mu(B)=0$ or $\mu(B)=1$.
 Denote $M(X, G)$ (resp. $M^e(X, G)$) by the collection of all $G$-invariant measures (resp. all ergodic measures) on $X$.
For $\mu \in M(X,G)$,  the \textit{support} of $\mu$ is defined by
$\supp(\mu )=\{x\in X\colon \mu (U)>0\text{ for any neighbourhood }U\text{ of
}x\}$.

\subsection{Relative topological entropy}

Let $G$ be a countable discrete infinite amenable group and $F(G)$ the set of all finite non-empty subsets
of $G$.

 Let $f : F(G)\ra \mathbb{R}$ be a function. We say that $f$ is

(1) monotone, if $f(E)\le f ( F )$ for any $E , F \in F ( G )$ with $E \subset F$;

(2) non-negative, if $f ( F )\ge 0$ for any $F \in  F ( G )$;

(3) $G$-invariant, if $f ( Fg )= f (F)$ for any $F \in F (G)$ and $g \in  G$;

(4) sub-additive, if $f( E \cup F )\le f (E)+ f (F) $ for any $E, F \in F ( G )$;

The following limit theorem for invariant sub-additive functions
on finite subsets of amenable groups is due to Lindenstrauss and Weiss \cite{LW00}.
It plays a central role in the definition of some dynamical invariants
such as topological entropy and measure-theoretic entropy.

\begin{lem}\cite[Theorem 6.1]{LW00}\label{lem 2.1} Let $G$ be a countable amenable group.
Let $f : F(G)\ra \mathbb{R}$ be a function be a monotone non-negative $G$-invariant
sub-additive function. Then for any F{\o}lner sequence $\{ F_m : m \in \mathbb{N} \}$ of G, the sequence
$\{ \frac{f(F_m)}{|F_m|}
: m\in \mathbb{N} \}$ converges and the value of the limit is independent of the choice of
the F{\o}lner sequence $\{F_m : m \in \mathbb{N} \}$.
 \end{lem}

For a given  TDS
$(X, G)$,  a {\it Borel measurable cover} (short for {\it cover}) of $X$ is defined as a finite collection of Borel measurable subsets of $X$ such that their union equals $X$.  A cover is called an {\it open cover} if every element of which is an open subset of $X$,  and is termed a {\it partition} if all its elements are pairwise disjoint. By convention, we  use $\mathcal{U}, \mathcal{V},\dots$ to denote covers and $\alpha, \beta,\dots$ to denote partitions of a TDS, and use $\mathcal{C}_X$,  $\mathcal{C}_X^o$ and $\mathcal{P}_X$ to denote the collection of all covers, open covers and partitions, respectively.

Given $\mathcal{U}, \mathcal{V} \in  \mathcal{C}_X$, we say that
$\mathcal{U}$ is  {\it finer} than $\mathcal{V}$ (denote by $\mathcal{V}\preceq \mathcal{U}$),
if each element of $\mathcal{U}$ is a subset of some element of $\mathcal{V}$. Define the {\it joining} of  $\mathcal{U}$ and $\mathcal{V}$ by $\mathcal{U}\vee\mathcal{V}=\{U\cap V: U\in \mathcal{U}, V\in \mathcal{V}\}$. It is clear that $\mathcal{U}\preceq \mathcal{U}\vee\mathcal{V}$ and $\mathcal{V}\preceq\mathcal{U}\vee\mathcal{V}$.

Now let $\pi: (X, G)\ra (Y, G)$ be a factor map between TDSs  and $\mathcal{U} \in  \mathcal{C}_X^o$.
For $E \subseteq X$, we define $N(\mathcal{U}, E)$ by the minimum  among all the cardinalities of the subsets
of $\mathcal{U}$  whose union covers $E$, and define
$N(\mathcal{U}|\pi) = \sup_{y\in Y}N(\mathcal{U}, \pi^{-1}(y))$.
 Let $H(\mathcal{U}|\pi):= \log N(\mathcal{U}|\pi)$.

 It is easy to see that $F \in  F ( G )\ra H(\bigvee_{g\in F}g^{-1}\mathcal{U}|\pi)$ is a monotone non-negative $G$-invariant sub-additive function. Thus by Lemma \ref{lem 2.1}
 we can define the {\it relative topological entropy of $G$ with respect to $\mathcal{U}$, $\pi$} by
$$
h_{\text{top}}(G,\mathcal{U}|\pi) = \lim_{m\to \infty}\frac{1}{|F_m|}H(\bigvee_{g\in F_m}g^{-1}\mathcal{U}|\pi),
$$
 and it is independent of the choice of F{\o}lner sequences.
 Moreover,  the {\it relative topological entropy of $G$ with respect to  $\pi$} is defined by
 $$
h_{\text{top}}(G|\pi)=\sup_{\mathcal{U}\in \mathcal{C}_X^o} h_{\text{top}}(T,\mathcal{U}|\pi).
$$

Blanchard et al \cite{B1993, B1995} introduced the notions of entropy pairs and entropy $\mu$-pairs for a measure,
from then on entropy pairs have been intensely studied by many researchers.
Following these ideas, Huang and Ye \cite{HY06} introduced the notions
of entropy tuples and entropy $\mu$-$n$-tuples for a measure.
Huang, Ye and Zhang \cite{HYZ07} studied relative entropy $n$-tuples and relative entropy $\mu$-$n$-tuples.

For actions of amenable groups,
Huang, Ye and Zhang \cite{HYZ11} studied entropy tuples and entropy $\mu$-$n$-tuples for a measure,
Dooley and Zhang \cite{DZ15} studied relative entropy $n$-tuples and relative entropy $\mu$-$n$-tuples.

Let $(x_1,x_2,\dots,x_n)\in X^{(n)}$. A cover  $\mathcal{U}=\{U_1,U_2,\dots,U_k\}$ of $X$ is called an {\it admissible cover with respect to $(x_1,x_2,\dots,x_n)$} if for each $1\leq j\leq k$ there is $1\leq i_j\leq n$ such that $x_{i_j}\notin\overline{U_j}$.
\begin{defn}\cite{HYZ07}\label{d1}
Let $\pi: (X, G)\ra (Y, G)$ be a factor map between TDSs and $2\le n\in\N$.
An $n$-tuple $(x_1,x_2,\dots,x_n)\in X^{(n)}\setminus \Delta_n(X)$ is called a {\it relative entropy $n$-tuple}
if for any admissible open cover $\mathcal{U}$ with respect to $(x_1,x_2,\dots,x_n)$,
we have $h_{\text{top}}(G,\mathcal{U}|\pi)>0$. If in addition, $x_i\neq x_j$ for each $1\leq i\neq j\leq n$,
we call it an {\it essential relative entropy $n$-tuple}.
\end{defn}	
We denote by $E_{n,\pi}(X,G)$ and $E_{n,\pi}^{e}(X,G)$ the set of all relative entropy $n$-tuples and the set of all essential relative entropy $n$-tuples, respectively.

\subsection{Relative entropy for a measure}
Given $\alpha\in \mathcal{P}_X$, $\mu\in
M(X,T)$ and a sub-$\sigma$-algebra
$\mathcal{C}\subseteq\mathcal{B}_X$,
let
\begin{equation*}
H_{\mu}(\alpha)=\sum_{A\in \alpha}-\mu(A)\log \mu(A)\,\,\text{and}\,\,
H_{\mu}(\alpha\mid\mathcal{C})=\sum_{A\in\alpha}\int_X-\mathbb{E}(1_A\mid\mathcal{C})\log
\mathbb{E}(1_A\mid\mathcal{C})\mathrm{d}\mu,
\end{equation*}
where $\mathbb{E}(1_A\mid\mathcal{C})$ represents the conditional expectation of the characteristic function $ 1_A$
with respect to $\mathcal{C}$.

Let $\pi: (X, G)\ra (Y, G)$ be a factor map between TDSs. It is easy to see that
 $\pi^{-1}\mathcal{B}_Y$ is a $G$-invariant sub-$\sigma$-algebra of
$\mathcal{B}_X$. For simplicity  we set  $H_\mu(\alpha|\pi)=H_\mu(\alpha|\pi^{-1}\mathcal{B}_Y)$.

It is easy to see that $F \in  F (G)\ra H(\bigvee_{g\in F}g^{-1}\alpha|\pi)$ is a monotone nonnegative $G$-invariant sub-additive function. Thus by Lemma \ref{lem 2.1}
We define the {\it relative $\mu$-measure  entropy of $G$ with respect to $\alpha$, $\pi$} by
$$h_{\mu}(G,\alpha|\pi) = \lim_{m\to \infty}\frac{1}{|F_m|}H(\bigvee_{g\in F_m}g^{-1}\alpha|\pi),.$$
The {\it relative $\mu$-measure  entropy of $G$ with respect to $\pi$}
is given by
$$h_{\mu}(G|\pi)=\sup_{\alpha \in \mathcal{P}_X} h_{\mu}(T,\alpha|\pi). $$

Let $(x_1,x_2,\dots,x_n)\in X^{(n)}$. A partition  $\alpha=\{A_1,A_2,\dots,A_k\}$ of $X$ is called an {\it admissible partition with respect to $(x_1,x_2,\dots,x_n)$} if for each $1\leq j\leq k$ there is $1\leq i_j\leq n$ such that $x_{i_j}\notin\overline{A_j}$.

\begin{defn}\cite{HYZ07}
Let $\pi: (X, G)\ra (Y, G)$ be a factor map between TDSs,  $\mu \in M(X,G)$ and $2\le n\in\N$.
An $n$-tuple $(x_1,x_2,\dots,x_n)\in X^{(n)}\setminus \Delta_n(X)$ is called a {\it relative entropy $\mu$-$n$-tuple}
if for any admissible partition $\alpha$ with respect to $(x_1,x_2,\dots,x_n)$,
we have $h_{\mu}(G,\alpha|\pi)>0$. If in addition, $x_i\neq x_j$ for each $1\leq i\neq j\leq n$,
we call it an {\it essential relative entropy $\mu$-$n$-tuple}.
\end{defn}	
We denote by $E_{n,\pi}^{\mu}(X,G)$ and $E_{n,\pi}^{\mu,e}(X,G)$
the set of all relative entropy $\mu$-$n$-tuples and the set of all essential relative entropy $\mu$-$n$-tuples, respectively.

\medskip
Let $\pi: (X, G)\ra (Y, G)$ be a factor map between TDSs and  $\mu \in M(X,G)$.
Define the {\it relative Pinsker $\sigma$-algebra} by
$$P_\mu(\pi^{-1}\mathcal{B}_Y)=\{A\in \mathcal{B}_X: h_{\mu}(G,\{A,A^c\}|\pi)= 0\}.$$
It is easy to see that $\pi^{-1}\mathcal{B}_Y\subseteq P_\mu(\pi^{-1}\mathcal{B}_Y)$
and $P_\mu(\pi^{-1}\mathcal{B}_Y)$ is a $G$-invariant sub-$\sigma$-algebra of $\mathcal{B}_X$.

Define the measure $\lambda_{n,\pi}^\mu$ on $\mathcal{B}_X^{(n)}$ with respect to $\pi$ by $$\lambda_{n,\pi}^\mu(\prod_{i=1}^nA_i)
=\int_{X}\prod_{i=1}^n\mathbb{E}(1_{A_i}|\mathcal{P}_\mu(\pi^{-1}\mathcal{B}_Y))\mathrm{d}\mu$$
for every $\prod_{i=1}^nA_i\in \mathcal{B}_X^{(n)}$.
It is easy to see that $\lambda_{1,\pi}^\mu=\mu$.

\begin{lem}\cite[Proposition 13.5]{DZ15}\label{lem:tuple-supp}
Let $\pi: (X, G)\ra (Y, G)$ be a factor map between TDSs and  $\mu \in M(X,G)$.
Then for any $2\le n\in \N$, we have
$$E_{n,\pi}^{\mu}(X,G)=\supp(\lambda_{n,\pi}^\mu)\setminus \Delta_n(X).$$
\end{lem}

\begin{lem}\cite[Lemma 6.3]{Y2015}\label{bl}
Let $\pi: (X, \mathcal{B}_X,\mu, G)\ra (Y, \mathcal{B}_Y, \nu, G)$ be a factor map between TDSs and  $\mu \in M^e(X,G)$.
Let $\mu=\int_X\mu_xd\mu(x)$ be the disintegration of $\mu$ over $P_\mu(\pi^{-1}\mathcal{B}_Y)$.
If $h_\mu(G|\pi)>0$, then the following holds:
\begin{itemize}
\item[{\rm 1)}] $\mu_x$  is non-atomic for $\mu$-a.e. $x\in X$;

\item[{\rm 2)}] $\phi: (X, \mathcal{B}_X,\mu, G)\ra (X, P_\mu(\pi^{-1}\mathcal{B}_Y),\mu, G)$ is a weakly mixing extension, i.e.
$\lambda_{2,\pi}^\mu$ is ergodic.
\end{itemize}
\end{lem}

\begin{lem}\label{lem:tuple-measure}
Let $\pi: (X, G)\ra (Y, G)$ be a factor map between TDSs and  $\mu \in M^e(X,G)$ with $h_{\mu}(G|\pi)>0$.
Then for any $n\ge 2$, $\lambda_{n,\pi}^\mu$ is ergodic and $\lambda_{n,\pi}^\mu(\Delta_n^\prime(X))=0$.
\end{lem}
\begin{proof}
(1) By Lemma \ref{bl}, $\phi: (X, \mathcal{B}_X,\mu, G)\ra (X, P_\mu(\pi^{-1}\mathcal{B}_Y),\mu, G)$ is a weakly mixing extension.
By \cite[Corollary 7.11]{Zimmer}, $\lambda_{n,\pi}^\mu$ is ergodic.

(2) By Fubini Theorem, we have $\lambda_{n,\pi}^\mu(\Delta_n^\prime(X))=\int_X \mu_x^{(n)}(\Delta_n^\prime(X))d\mu(x)=0$.
\end{proof}

\section{Ergodic decompositions}\label{sect:Ergodic decompositions}

In \cite{Y2015}, Yan gave the ergodic decompositions of  measure conditional entropy of partition.
We will study the ergodic decompositions of  measure conditional entropy of cover.

\begin{lem}\cite[Theorem 3.3]{Y2015}\label{20260621}
Let $\pi:(X,G)\ra (Y,G)$ be a factor map between TDSs.
Let $\mu\in M(X,G)$ and $\mu=\int_{\Omega}\theta \textrm{d} m(\theta)$ be the ergodic decomposition of $\mu$.
Then $$h_\mu(G,\alpha|\pi)=\int_{\Omega}h_{\theta}(G,\alpha|\pi)\textrm{d} m(\theta),
h_\mu(G|\pi)=\int_{\Omega}h_{\theta}(G|\pi)\textrm{d} m(\theta)$$
\end{lem}

For $\mu \in \mathcal{M}(X,G)$ we define the conditional entropy of $\mathcal{U}\in
\mathcal{C}_X$ with respect to a $G$-invariant
sub-$\sigma$-algebra $\mathcal{C}$ of  $\mathcal{B}_X^\mu$:

$$h_{\mu}(G, \mathcal{U}| \mathcal{C})= \inf_{\alpha\in \mathcal{P}_X, \alpha\succeq \mathcal{U}} h_\mu (G, \alpha| \mathcal{C}).$$

\begin{lem}\label{er-decom}
Let $\pi:(X,G)\ra (Y,G)$ be a factor map between TDSs.
Let $\mu\in M(X,G)$ and $\mu=\int_{\Omega}\theta \textrm{d} m(\theta)$ be the ergodic decomposition of $\mu$.
Then $$h_\mu(G,\mathcal{U}|\pi)=\int_{\Omega}h_{\theta}(G,\mathcal{U}|\pi)\textrm{d} m(\theta)$$
\end{lem}
\begin{proof}
Let $\mathcal{U} =\{U_1,U_2,\cdots,U_M\}\in \mathcal{C}_X$ and define
$$\mathcal{U}^\ast=\{\alpha\in  \mathcal{P}_X: \alpha=\{A_1,A_2,\cdots,A_M\}, A_m\subset U_m, 1\le m\le M\}$$

Since $X$ is a compact metric space,
there exists a sequence of partitions $(\alpha_k : k\in \mathbb{N})$ in $\mathcal{U}^\ast$
which is $L^1(X,\mathcal{B}_X,\nu)$-dense in $\mathcal{U}^\ast$ for every $\nu \in M(X,G)$.
In particular for every $\nu \in M(X,G)$ we have
$$h_\mu(G,\mathcal{U}|\pi)=\inf_{n\in \mathbb{N}}h_\mu(G,\alpha_k|\pi).$$
Write $\alpha_k=\{A_1^k,\cdots, A_M^k\}, k\in \mathbb{N}$, by Fatou’s lemma and Lemma \ref{20260621}
\begin{equation*}
    \begin{aligned}
		&h_\mu(G,\mathcal{U}|\pi)=\inf_{n\in \mathbb{N}}h_\mu(G,\alpha_k|\pi)
=\inf_{n\in \mathbb{N}}\int_{\Omega}h_{\theta}(G,\alpha_k|\pi)\textrm{d} m(\theta)\\
\ge & \int_{\Omega}\inf_{n\in \mathbb{N}}h_{\theta}(G,\alpha_k|\pi)\textrm{d} m(\theta)
=\int_{\Omega}h_{\theta}(G,\mathcal{U}|\pi)\textrm{d} m(\theta)
    \end{aligned}
\end{equation*}

To prove the other one, for every $\epsilon>0$ and $n \in \mathbb{N}$
define $$B_n^\epsilon
=\{\theta \in \Omega : h_{\theta}(G,\alpha_n|\pi)< h_{\theta}(G,\mathcal{U}|\pi)+ \epsilon\}.$$
we know that $m(\Omega\Delta\cup_{n\in \mathbb{N}} B_n^\epsilon)=0$,
so there exists a partition $(\Omega_n : n\in \mathbb{N})$ of $\Omega$ with $m(\Omega_n)>0$,
and a sequence of partitions $(\alpha_{k_n}: n\in \mathbb{N})$ such that for
m-a.e. $\theta \in \Omega_n$ we have
$$h_{\theta}(G,\alpha_{k_n}|\pi)< h_{\theta}(G,\mathcal{U}|\pi)+ \epsilon.$$

Therefore, for every $n\in  \mathbb{N}$ we can define
$\mu_n \in M(X,G)$ as $$\mu_n=\frac{1}{m(\Omega_n)}\int_{\Omega_n}\theta \textrm{d} m(\theta),$$
then $\mu=\sum_{n\in \mathbb{N}}m(\Omega_n)\mu_n$.

We deduce
\begin{equation*}
    \begin{aligned}
		&h_{\mu_n}(G,\alpha_{k_n}|\pi)
=\frac{1}{m(\Omega_n)}\int_{\Omega_n}h_{\theta}(G,\alpha_{k_n}|\pi)\textrm{d} m(\theta)\\
 \le &\frac{1}{m(\Omega_n)}\int_{\Omega_n}h_{\theta}(G,\mathcal{U} |\pi)\textrm{d} m(\theta)+\epsilon
    \end{aligned}
\end{equation*}

Since $\{\mu_n:n\in \mathbb{N}\}$ is mutually singular,
there exists a family of Borel sets $\{X_n: n\in \mathbb{N}\}$ such that
$\mu_n(X_n) = 1$ and $\mu_n(X_k) = 0$ for $k\not= n$.
For $1\le i \le M$ define $A_i = \cup_{n\ge 1}(X_n \cap A_i^{k_n})$
then $\alpha =\{A_1,\cdots ,A_M\}\in \mathcal{U}^\ast$

\begin{equation*}
    \begin{aligned}
		&h_\mu(G,\mathcal{U}|\pi)\le h_\mu(G,\alpha|\pi)=\sum_{n\in \mathbb{N}} m(\Omega_n)h_{\mu_n}(G,\alpha|\pi)\\
 = & \sum m(\Omega_n)h_{\mu_n}(G,\alpha_{k_n}|\pi)\le\int_{\Omega}h_{\theta}(G,\mathcal{U}|\pi)\textrm{d} m(\theta)+\epsilon
    \end{aligned}
\end{equation*}

\end{proof}

The following lemma \cite{DZ15} is fundamental in the study of
the structure of relative entropy $n$-tuples.

\begin{lem} \cite[Theorem 3.11]{DZ15}\label{etfm-cor-6}
Let $\mu\in M(X,G)$ and $\mathcal{U}
 =\{U_1,\cdots, U_n \}\in \mathcal{C}_X$. Then the following
 statements are equivalent:
\begin{enumerate}

\item $h_{\mu} (G,\mathcal{U}|\pi)>0$;

\item $h_\mu (G, \alpha)> 0$ if $\alpha\in \mathcal{C}_X$ is
finer than $\mathcal{U}$;

\item
$\lambda_{n,\pi}^\mu(\prod_{i=1}^n U_i^c)>0$.
\end{enumerate}
\end{lem}

\begin{thm} \label{et-decom}
Let $\pi:(X,G)\ra (Y,G)$ be a factor map between TDSs and $2\le n\in\N$.  We have the following characterizations.
Let $\mu\in M(X,G)$ and $\mu=\int_{\Omega}\theta \textrm{d} m(\theta)$ be the ergodic decomposition of $\mu$. Then

\begin{enumerate}
\item for $m$-a.e. $\theta\in \Omega$, $E_{n,\pi}^{\theta}(X,
G)\subseteq E_{n,\pi}^{\mu}(X, G)$ for each $n\ge 2$.

\item if $(x_i)_1^n\in E_n^{\mu}(X, G)$, then for every
measurable neighborhood $V$ of $(x_i)_1^n$, $m (\{ \theta\in \Omega:
V\cap E_{n,\pi}^{\theta}(X, G)\not=\emptyset\})>0$. Thus for an
appropriate choice of $\Omega'\subset \Omega$, we can require
$$
\overline{\cup \{ E_{n,\pi}^{\theta}(X, G): \theta \in \Omega'
\}}\setminus \Delta_n(X) =E_{n,\pi}^{\mu}(X, G).$$
\end{enumerate}
\end{thm}
\begin{proof}

(1)
Let $U_i,\ i=1,\cdots,n$ be open subsets of $X$ with
$\bigcap_{i=1}^n \overline{U_i}=\emptyset$ and $(\prod_{i=1}^n
U_i) \cap E^\mu_{n,\pi}(X,G)=\emptyset$.  By Lemma
\ref{lem:tuple-supp},
$\lambda_{n,\pi}^\mu(\prod_{i=1}^n U_i)= 0$. By Lemma
\ref{etfm-cor-6} $h_\mu(G,\mathcal{U}|\pi)=0$ , where $\mathcal{U}=\{ U_1^c,\cdots,U_n^c\}$. As
$\int_\Omega h_{\theta}(G,\mathcal{U}|Y) d
m(\theta)=h_\mu(G,\mathcal{U}|\pi)=0$, for
$m$-a.e. $\theta\in \Omega$ $h_{\theta}(G,\mathcal{U}|\pi)=0$.
By Lemma \ref{etfm-cor-6} $\lambda_{n,\pi}^{\theta}(\prod_{i=1}^n U_i)=0$.
By Lemma \ref{lem:tuple-supp} and the assumption of $\bigcap_{i=1}^n \overline{U_i}=\emptyset$,
we have $(\prod_{i=1}^n U_i) \cap E^{\theta}_{n,\pi}(X,G)=\emptyset$.

Since $E_{n,\pi}^\mu(X,G)\cup \Delta_n(X)\subseteq R_{n,\pi}$ is closed, its
complement can be written as a union of countable sets of the form
$\prod_{i=1}^n U_i$ with $U_i,i=1,\cdots,n$ open subsets satisfying
$\bigcap_{i=1}^n \overline{U_i}=\emptyset$. Then applying the above
procedure to each such a subset $\prod_{i=1}^n U_i$ one has that for
$m$-a.e. $\theta\in \Omega$, $E^{\theta}_n(X,G)\cap
{(E^\mu_{n,\pi}(X,G))}^c=\emptyset$, equivalently,
$E^{\theta}_{n,\pi}(X,G)\subseteq E^\mu_n(X,G)$.

(2)  With no loss of generality we assume $V=\prod_{i=1}^n A_i$, where $A_i$ is a
closed neighborhood of $x_i$, $1\le i\le n$ and $\bigcap_{i=1}^n
A_i=\emptyset$. By Lemma \ref{lem:tuple-supp} we have $\lambda_{n,\pi}^\mu(\prod_{i=1}^n A_i)>0$.
Using Lemma \ref{er-decom} and
Lemma \ref{etfm-cor-6}, we have
\begin{equation*}
\int_\Omega h_{\theta}(T,\{ A_1^c,\cdots,A_n^c \}|\pi) d m(\theta)
=h_\mu(T,\{ A_1^c,\cdots,A_n^c \}|\pi)>0.
\end{equation*}
There exists $\Omega'\subseteq \Omega$ with $m(\Omega')>0$ such that
if $\theta\in \Omega'$ then
$h_{\theta}(G,\{ A_1^c,\cdots,A_n^c \}|\pi)>0$.
By Lemma \ref{etfm-cor-6} $\lambda_{n,\pi}^{\theta}\left(\prod_{i=1}^nA_i\right)>0$.
By Lemma
\ref{lem:tuple-supp} $(\prod_{i=1}^nA_i) \cap E_{n,\pi}^{\theta}(X,G)\not=\emptyset$,
 i.e. $m(\{\theta\in \Omega: V\cap
E_{n,\pi}^{\theta}(X,G)\not=\emptyset\})>0$.

\end{proof}

\section{relative mean sensitivity and relative entropy for a measure}\label{sect:MS-tuple:measure}

In this section, we examine the relationship between relative mean sensitivity and relative entropy for a measure.
Let $G$ be a countable discrete infinite amenable group with $\mathcal{F}=\{F_m\}_{m=1}^\infty$ a
F{\o}lner sequence of $G$.

Let $\pi:(X,G)\ra (Y,G)$ be a factor map between TDSs with $\mu\in M(X,G)$ and $2\le n\in\N$.
We say that $\pi$ is {\it relative mean $\mu$-$n$-sensitive} with respect to $\mathcal{F}$
if there is $\ep>0$ such that for any $A\in \mathcal{B}_X$ with $\mu(A)>0$
there are $n$ distinct points $x_1,x_2,\dotsc,x_n\in A$
with $\pi(x_1)=\pi(x_2)=\cdots =\pi(x_n)$ such that
$$
\limsup_{m\to\infty}\frac{1}{|F_m|}\sum_{g\in F_m}\min_{1\le i\neq j\le n} d(g x_i, g x_j)>\ep.
$$

The following proposition is a simple observation of relative mean sensitivity for a measure.

\begin{prop}\label{prop:eq3}
Let $\pi:(X,G)\ra (Y,G)$ be a factor map and $2\le n\in\N$.
Then the following conditions are equivalent:
\begin{enumerate}
\item $\pi$ is relative mean $\mu$-$n$-sensitive  with respect to $\mathcal{F}$;
\item there is $\delta>0$ such that for any $A\in \mathcal{B}_X$ with $\mu(A)>0$
there are $(x_1,x_2,\dots,x_n)\in A^{(n)}\cap R_{n,\pi}$
and a subset $F$ of $G$ with $\overline{D}_{\mathcal{F}}(F)>\delta$ such that
$d(gx_i,gx_j)>\delta$
for all $1\leq i<j\leq n$ and $g\in F$.
\end{enumerate}
\end{prop}
\begin{proof}
(1)$\Rightarrow$(2) Since $X$ is compact metric space, without loss of generality,
we assume that $d(x,y) \leq 1$ for any $x,y \in X$.
Since $\pi$ is relative mean $\mu$-$n$-sensitive with respect to $\mathcal{F}$,
there exists $\ep>0$ such that
for any $A\in \mathcal{B}_X$ with $\mu(A)>0$ there are $n$ distinct points $x_1,x_2,\dotsc,x_n\in A$
with $\pi(x_1)=\pi(x_2)=\cdots =\pi(x_n)$ and

\[
	\limsup_{m\to\infty}\frac{1}{|F_m|}\sum_{g\in F_m}\min_{1\le i\neq j\le n} d(g x_i, g x_j)>\ep
.\]
Let $\delta=\frac{\ep}{2}$ and
$F = \{ g\in G: \min_{1 \leq i \neq j \leq n} d(g x_{i}, g x_{j}) > \delta \}.$
Since
\begin{equation*}
    \begin{aligned}
		2\delta=\ep <&\limsup_{m\to\infty}\frac{1}{|F_m|}\sum_{g\in F_m}\min_{1\le i\neq j\le n} d(g x_i, g x_j)\\
= &\limsup_{m\to\infty}  \frac{1}{|F_m|} \left[  \sum_{g\in F_m\cap F}\min_{1\le i\neq j\le n} d(g x_i, g x_j)
+  \sum_{g\in F_m \setminus F}\min_{1\le i\neq j\le n} d(g x_i, g x_j) \right] \\
\leq &\limsup_{m\to\infty}  \frac{1}{|F_m|} \left[ \sum_{g\in F_m\cap F} 1 + \sum_{F_m \setminus F} \delta \right] \leq \overline{D}_{\mathcal{F}}(F) + \delta,
    \end{aligned}
\end{equation*}
thus $\overline{D}_{\mathcal{F}}(F)> \delta.$

(2)$\Rightarrow$(1) There is $\delta>0$ such that for any $A\in \mathcal{B}_X$ with $\mu(A)>0$
there are $(x_1,\dots,x_n)\in A^{(n)}\cap R_{n,\pi}$
and a subset $F$ of $G$ with $\overline{D}_{\mathcal{F}}(F)>\delta$ such that
$d(gx_i,gx_j)>\delta$
for all $1\leq i<j\leq n$ and $g\in F$.

Let $\epsilon=\delta^2$. Then
\begin{equation*}
\begin{aligned}
	\limsup_{m\to\infty}\frac{1}{|F_m|}\sum_{g\in F_m}\min_{1\le i\neq j\le n} d(g x_i, g x_j)
& \geq
\limsup_{m\to\infty}\frac{1}{|F_m|}\sum_{g\in F_m\cap F}\min_{1\le i\neq j\le n} d(g x_i, g x_j) \\
 & > \lim_{m\to\infty} \sup \frac{1}{|F_m|} \sum_{g\in F_m\cap F} \delta \\
 & = \overline{D}_{\mathcal{F}}(F) \cdot \delta
 > \delta^{2}=\epsilon.
\end{aligned}
\end{equation*}

\end{proof}

Garc\'{\i}a-Ramos \cite{G17} introduced the notion mean $\mu$-sensitive pairs, and Li-Yu \cite{LY21}
studied its multi-variant version. We will introduce and study the relative mean $\mu$-sensitive tuples.

\begin{defn}
Let $\pi:(X,G)\ra (Y,G)$ be a factor map between TDSs with $\mu\in M(X,G)$ and $2\le n\in\N$. We say  that $(x_1,x_2,\dots,x_n)\in X^{(n)}\setminus \Delta_n(X)$ is a {\it relative mean $\mu$-$n$-sensitive tuple}
with respect to $\mathcal{F}$
if for any open neighbourhoods $U_i$ of $x_i$ with $i=1,2,\dots,n$, there is $\delta> 0$ such that for any $A\in \mathcal{B}_X$ with $\mu(A)>0$ there are $y_1,y_2,\dots,y_n\in A$ with $\pi(y_1)=\pi(y_2)=\cdots =\pi(y_n)$
and a subset $F$ of $G$ with $\overline{D}_{\mathcal{F}}(F)>\delta$ such that $g y_i \in U_i$ for all $i=1,2,\dots,n$ and $g\in F$.
If in addition, $x_i\neq x_j$ for each $1\leq i\neq j\leq n$,
we call it an {\it essential relative mean $\mu$-$n$-sensitive tuple} with respect to $\mathcal{F}$.
\end{defn}

We denote by $MS_{n,\pi}^\mu(X,G)$ and $MS_{n,\pi}^{\mu,e}(X,G)$ the set of all relative mean $\mu$-$n$-sensitive tuple
with respect to $\mathcal{F}$ and the set of all essential relative mean $\mu$-$n$-sensitive tuple with respect to $\mathcal{F}$, respectively.

\begin{prop}\label{prop:mu-MS=mu-MS-tuple}
Let $\pi:(X,G)\ra (Y,G)$ be a factor map between TDSs with $\mu\in M^e(X,G)$ and $2\le n\in\N$.
Then the following statements are equivalent:
\begin{enumerate}
\item[(1)] $\pi$ is relative mean $\mu$-$n$-sensitive with respect to $\mathcal{F}$;
\item[(2)] $MS_{n,\pi}^{\mu,e}(X,G)\neq\emptyset$.
\end{enumerate}
\end{prop}

\begin{proof}

(1)$\Rightarrow$(2)
Since $\pi$ is relative mean $\mu$-$n$-sensitive with respect to $\mathcal{F}$, by Proposition \ref{prop:eq3}
there is $\delta>0$ such that
for any $A\in \mathcal{B}_X$ with $\mu(A)>0$
there are $(x_1,x_2,\dots,x_n)\in A^{(n)}\cap R_{n,\pi}$
and a subset $F$ of $G$ with $\overline{D}_{\mathcal{F}}(F)>\delta$ such that
$d(gx_i,gx_j)>\delta$ for all $1\leq i<j\leq n$ and $g\in F$.

Let
$$
R_{n,\pi}^\delta=\{ (x_1,x_2,\dots,x_n)\in R_{n,\pi}: \min_{1\le i<j\le n}
d(x_i,x_j)\ge \delta\}.
$$
From the continuity of $d$, $R_{n,\pi}^\delta$ is closed in $R_{n,\pi}$. If $MS_{n,\pi}^{\mu,e}(X, G)=\emptyset$, then for any $\overline{x}=(x_1,x_2,\dots,x_n)\in R_{n,\pi}^\delta$, there exist open neighborhoods
$U(\overline{x})$ of $\overline{x}$  such that for any $b>0$
there is a compact subset $A_b(\overline{x})\subseteq X$ with $\mu(A_b(\overline{x}))>0$ and
$$\overline{D}_{\mathcal{F}}(\{g\in G: (gz_1,gz_2,\dots,gz_n)\in
U(\overline{x})\})\leq b,$$
for any $z_1,z_2,\dots,z_n\in A_b(\overline{x})$ with $\pi(z_1)=\pi(z_2)=\cdots =\pi(z_n)$.

Since $\{ U(\overline{x}):\overline{x}\in R_{n,\pi}^\delta\}$ is
an open cover of $R_{n,\pi}^\delta$,  there exist
$\overline{y_1},\overline{y_2},\dots,\overline{y_k}\in
R_{n,\pi}^\delta$ such that
$\cup_{j=1}^k  U(\overline{y_j})\supseteq
R_{n,\pi}^\delta.$
Let $b_0=\frac{\delta}{k}$.
Since $\mu$ is ergodic, there exist $g_1,g_2\cdots,g_k\in G$ such that $A=\cap_{j=1}^k
g_jA_{b_0}(\overline{y_j})$ is compact with $\mu(A)>0$. Then for any $1\le j\le k$ and
any $z_1,z_2,\dots,z_n\in A$ with $\pi(z_1)=\pi(z_2)=\cdots =\pi(z_n)$,
\begin{align*}
&(\{g\in G:(gz_1,gz_2,\dots,gz_n)\in U(\overline{y_j})\})\\
&=(\{g\in G:gg_j(g_j^{-1}z_1,g_j^{-1}z_2,\dots,g_j^{-1}z_n)\in U(\overline{y_j})\})\\
&=(\{h\in G:h(g_j^{-1}z_1,g_j^{-1}z_2,\dots,g_j^{-1}z_n)\in U(\overline{y_j})\})g_j^{-1}
\end{align*}

By Lemma \ref{20260620} we have
\begin{align*}
&\overline{D}_{\mathcal{F}}(\{g\in G:(gz_1,gz_2,\dots,gz_n)\in
U(\overline{y_j})\})\\
&=\overline{D}_{\mathcal{F}}(\{h\in G:h(g_j^{-1}z_1,g_j^{-1}z_2,\dots,g_j^{-1}z_n)\in
U(\overline{y_j})\})\\
&\leq b_0=\frac{\delta}{k}.
\end{align*}

Since $\pi$ is relative mean $\mu$-$n$-sensitive,
by Proposition \ref{prop:eq3}
there exist $v_1,v_2,\dots,v_n\in A$ with $\pi(v_1)=\pi(v_2)=\cdots =\pi(v_n)$ such that
$$\overline{D}_{\mathcal{F}}(\{g\in G:(gv_1,gv_2,\dots,gv_n)\in
R_{n,\pi}^\delta \})> \delta$$
and so
$$
\overline{D}_{\mathcal{F}}(\{g\in G:(gv_1,gv_2,\dots,gv_n)\in
\cup_{j=1}^k  U(\overline{y_j})\})> \delta.
$$
This implies that there is  $1\le j\le k$ such that
$$
\overline{D}_{\mathcal{F}}(\{g\in G:(gv_1,gv_2,\dots,gv_n)\in
U(\overline{y_j})\})> \frac{\delta}{k},
$$
a contradiction. Hence $MS_{n,\pi}^{\mu,e}(X,G)\neq\emptyset$.

(2)$\Rightarrow$(1)
Let $(x_1,\dots,x_n)\in MS_{n,\pi}^{\mu,e}(X,G)$ and $0<\delta_1=\min_{1 \leq i \neq j \leq m} d(x_{i}, x_{j})$.
There exist open neighborhood $\prod_{i=1}^n U_i$ of $(x_1,\dots,x_n)$
such that $\min_{1\le i\neq j\le n}d(U_i,U_j)>\frac{1}{2}\delta_1$.
Since $(x_1,\dots,x_n)\in MS_{n,\pi}^{\mu,e}(X,G)$,
there exists $\delta_2>0$ such that for any $A\in \mathcal{B}_X$ with $\mu(A)>0$ there are $y_1,y_2,\dots,y_n\in A$ with $\pi(y_1)=\pi(y_2)=\cdots =\pi(y_n)$
and a subset $F$ of $G$ with $\overline{D}_{\mathcal{F}}(F)>\delta_2$ such that $g y_i \in U_i$ for all $i=1,2,\dots,n$ and $k\in F$.
Let $\delta=\min\{\frac{1}{2}\delta_1,\delta_2\}$.
Then $\min_{1\le i\neq j\le n}d(g y_i,g y_j)\ge\min_{1\le i\neq j\le n}d(U_i,U_j)> \delta$ for all $g\in F$
with $\overline{D}_{\mathcal{F}}(F)>\delta$.
By Proposition \ref{prop:eq3} $\pi$ is relative mean $\mu$-$n$-sensitive with respect to $\mathcal{F}$.

\end{proof}

Next we explore the relation between relative entropy tuples and relative mean sensitive tuples for a measure.

\begin{thm}\label{thm:entr-tuple--ms-tuple}
Let $\pi:(X,G)\ra (Y,G)$ be a factor map between TDSs with $\mu\in M^e(X,G)$ and $2\le n\in\N$.
Then $E_{n,\pi}^{\mu}(X,G)\subset MS_{n,\pi}^{\mu}(X,G)$.
\end{thm}

\begin{proof}
If $h_{\mu}(G|\pi)=0$, then $\emptyset=E_{n,\pi}^{\mu}(X,G)\subset MS_{n,\pi}^{\mu}(X,G)$.

We now assume $h_{\mu}(G|\pi)>0$.
Define the measure $\lambda_{n,\pi}^\mu$ on $\mathcal{B}_X^{(n)}$ with respect to $\pi$ by $$\lambda_{n,\pi}^\mu(\prod_{i=1}^nA_i)
=\int_{X}\prod_{i=1}^n\mathbb{E}(1_{A_i}|\mathcal{P}_\mu(\pi^{-1}\mathcal{B}_Y))(x)\mathrm{d}\mu(x)$$
for every $\prod_{i=1}^nA_i\in \mathcal{B}_X^{(n)}$,
where $P_{\mu}(\pi^{-1}\mathcal{B}_Y)$ is the relative Pinsker $\sigma$-algebra.

Since $h_{\mu}(G|\pi)>0$ and $\mu\in M^e(X,G)$,
by Lemma \ref{lem:tuple-supp} and  Lemma \ref{lem:tuple-measure},
$\lambda_{n,\pi}^\mu$ is an ergodic measure on $X^{(n)}$ and
$\supp(\lambda_{n,\pi}^\mu)\setminus \Delta_n(X) = E_{n,\pi}^{\mu}(X,G)\subset R_{n,\pi}$.
Let $(x_1,\dots, x_n)\in E_{n,\pi}^{\mu}(X,G)$, then $(x_1,\dots, x_n)\in \supp(\lambda_{n,\pi}^\mu)$.
For any $\tau>0$, we have $\lambda_{n,\pi}^\mu(\prod_{i=1}^n B(x_i,\tau))>0$.

Since $\lambda_{n,\pi}^\mu$ is ergodic, by the Lindenstrauss pointwise ergodic theorem \cite{L2001},
for $\lambda_{n,\pi}^\mu$-a.e. $(y_1,\dots, y_n)\in X^{(n)}$,
$$
\lim_{m\to \infty}\frac{1}{|F_m|}|\{g\in F_m:(gy_1,\cdots, gy_n)\in \prod_{i=1}^n B(x_i,\tau)\}|
=\lambda_{n,\pi}^\mu(\prod_{i=1}^n B(x_i,\tau)).
$$
For any $A\in \mathcal{B}_X$ with $\mu(A)>0$,
we have $$\mu(A)=\lambda_{1,\pi}^\mu(A)=\int_X\mathbb{E}(1_{A}|\mathcal{P}_\mu(\pi^{-1}\mathcal{B}_Y))(x)\textrm{d}\mu(x)>0.$$
From the definition of conditional expectation,
for $\mu$-a.e. $\mathbb{E}(1_{A}|\mathcal{P}_\mu(\pi^{-1}\mathcal{B}_Y))\ge 0$.
Let $B=\{x\in X:\mathbb{E}(1_{A}|\mathcal{P}_\mu(\pi^{-1}\mathcal{B}_Y))(x)>0\}$.
Then $\mu(B)>0$, we have
$$\lambda_{n,\pi}^\mu(A^{(n)})
=\int_{X}[\mathbb{E}(1_{A}|\mathcal{P}_\mu(\pi^{-1}\mathcal{B}_Y))(x)]^n\mathrm{d}\mu(x)
\ge \int_{B}[\mathbb{E}(1_{A}|\mathcal{P}_\mu(\pi^{-1}\mathcal{B}_Y))(x)]^n\mathrm{d}\mu(x)>0.$$
Since $\supp(\lambda_{n,\pi}^\mu) \subset E_{n,\pi}^{\mu}(X,G)\cup \Delta_n(X)\subset R_{n,\pi}$, we have
$\lambda_{n,\pi}^\mu(A^{(n)}\cap R_{n,\pi})>0$.
So there are $y_1^\prime,\dots, y_n^\prime\in A$  with $\pi(y_1')=\pi(y_2')=\cdots =\pi(y_n')$ such that
$$\lim_{m\to \infty}\frac{1}{|F_m|}|\{g\in F_m:gy_i'\in B(x_i,\tau), i=1,2,\dots,n\}|
=\lambda_{n,\pi}^\mu(\prod_{i=1}^n B(x_i,\tau))>0.$$
That implies that $(x_1,\dots, x_n)\in MS_{n,\pi}^{\mu}(X,G)$, completing the proof.
\end{proof}

\begin{cor}
Let $\pi:(X,G)\ra (Y,G)$ be a factor map between TDSs with $\mu\in M^e(X,G)$ and $h_\mu(G| \pi)>0$.
Then $\pi$ is relative mean $\mu$-$n$-sensitive for every $2\le n\in\N$.
\end{cor}
\begin{proof}
Since $h_\mu(G| \pi)>0$ and $\mu\in M^e(X,G)$,
by Lemma \ref{lem:tuple-measure} $E_{n,\pi}^{\mu,e}(X,G)\not=\emptyset$ for any $2\le n\in\N$.
Since $\mu\in M^e(X,G)$, by Theorem \ref{thm:entr-tuple--ms-tuple} $MS_{n,\pi}^{\mu,e}(X,G)\not=\emptyset$.
By Proposition \ref{prop:mu-MS=mu-MS-tuple} $\pi$ is relative mean $\mu$-$n$-sensitive for every $2\le n\in\N$.

\end{proof}

\section{relative mean sensitivity and relative entropy}\label{sect:MS-tuple:topological}

In this section we build the relationship between relative mean sensitivity and relative entropy.

Let $\pi:(X,G)\ra (Y,G)$ be a factor map and $2\le n\in\N$.
We say that $\pi$ is {\it relative mean $n$-sensitive} with respect to $\mathcal{F}$ if there is $\ep>0$ such that
any non-empty open subset $U$ of $X$ there are $n$ distinct points $x_1,x_2,\dots,x_n\in U$
with $\pi(x_1)=\pi(x_2)=\cdots =\pi(x_n)$ and
$$
\limsup_{m\to\infty}\frac{1}{|F_m|}\sum_{g\in F_m}\min_{1\le i\neq j\le n} d(g x_i, g x_j)>\ep.
$$

The following proposition is an observation concerning the relative mean sensitivity. As the proof of this observation follows a similar approach to Proposition \ref{prop:eq3}, we omit the details.

\begin{prop}\label{prop:eq4}
Let $\pi:(X,G)\ra (Y,G)$ be a factor map and $2\le n\in\N$.
Then the following conditions are equivalent:
\begin{enumerate}
  \item $\pi$ is relative mean $n$-sensitive with respect to $\mathcal{F}$;
  \item there is $\delta>0$ such that for any non-empty open subset $U$ of $X$
there are $(x_1,\dots,x_n)\in U^{(n)}\cap R_{n,\pi}$
and a subset $F$ of $G$ with $\overline{D}(F)>\delta$ such that
$d(gx_i,gx_j)>\delta$
for all $1\leq i<j\leq n$ and $g\in F$.
\end{enumerate}
\end{prop}

Li, Ye and Yu \cite{LYY22} introduced the notion mean sensitive tuples.
We will introduce and study the relative mean sensitive tuples.

\begin{defn}
Let $\pi:(X,G)\ra (Y,G)$ be a factor map and $2\le n\in\N$.
We say  that $(x_1,x_2,\dots,x_n)\in X^{(n)}\setminus \Delta_n(X)$
is a \textit{relative mean $n$-sensitive tuple} with respect to $\mathcal{F}$ if for any open neighbourhoods $U_i$ of $x_i$ with $i=1,2,\dots,n$,
there is $\delta> 0$ such that for any non-empty open subset $U$ of $X$,
 there are $(y_1,y_2,\dots,y_n)\in U^{(n)}\cap R_{n,\pi}$
and a subset $F$ of $G$ with $\overline{D}_{\mathcal{F}}(F)>\delta$ such that $g y_i \in U_i$ for all $i=1,2,\dots,n$ and $g\in F$.
If in addition, $x_i\neq x_j$ for each $1\leq i\neq j\leq n$,
we call it an {\it essential relative mean $n$-sensitive tuple}.
\end{defn}

We denote by $MS_{n,\pi}(X,G)$ and $MS_{n,\pi}^e(X,G)$ the set of all relative mean $n$-sensitive tuple
and the set of all essential relative mean $n$-sensitive tuple, respectively.

\begin{prop}\label{thm:ms=ms-tuple}
Let $\pi:(X,G)\ra (Y,G)$ be a factor map between transitive TDSs with $2\le n\in \N$.
Then the following conditions are equivalent:
\begin{enumerate}
  \item $\pi$ is relative mean $n$-sensitive with respect to $\mathcal{F}$;
  \item $MS_{n,\pi}^e(X,G)\neq \emptyset$.
\end{enumerate}
\end{prop}
\begin{proof}
(1)$\Rightarrow$(2)
Since $\pi$ is relative mean $n$-sensitive,
by Proposition \ref{prop:eq4}
for any non-empty open subset $U$ of $X$ there exist $x_1,\dots,x_n\in U$ with $\pi(x_1)=\pi(x_2)=\cdots=\pi(x_n)$
and a subset $F$ of $G$ with $\overline{D}_{\mathcal{F}}(F)>\delta$ such that
$\underset{ 1\leq i\not=j \leq n}{\min}d(gx_i,gx_j)>\delta$
for all $g\in F$.
Put
$$
R_{n,\pi}^\delta=\{ (x_1,x_2,\dots,x_n)\in R_{n,\pi}: \min_{1\le i<j\le n}
d(x_i,x_j)\ge \delta\}.
$$
From the continuity of $d$, $R_{n,\pi}^\delta$ is a compact subset of $R_{n,\pi}$.
Let $x\in X$ be a transitive point. For each $m\in\N$,  denote $W_m=B(x,\frac{1}{m})$.
Then there are $x_m^1,\dots,x_m^n\in W_m$ with $\pi(x_m^1)=\pi(x_m^2)=\cdots=\pi(x_m^n)$ such that
$$
\overline{D}_{\mathcal{F}}(\{g\in G: (gx_m^1,\dots, gx_m^n)\in R_{n,\pi}^\delta\})>\delta.
$$
Since $R_{n,\pi}^\delta$ is a closed set,
$R_{n,\pi}^\delta$ can be expressed as a finite union of non-empty closed sets of diameter less than $1$,
i.e., $R_{n,\pi}^\delta = \cup_{i=1}^{N_1}A_1^i$ and $\diam(A_1^i)<1$.
Then for each $m\in \N$ there is $1\leq N_{1,m}\leq N_1$ such that
$$\overline{D}_{\mathcal{F}}(\{g\in G: (gx_m^1,\dots, gx_m^n)\in A_1^{N_{1,m}} \})> \frac{\delta}{N_1}.$$
Without loss of generality we assume $N_{1,m}=1$ for all $m\in \N$ and so
$$
\overline{D}_{\mathcal{F}}(\{g\in G: (gx_m^1,\dots, gx_m^n)\in A_1^{1}\}) >\frac{\delta}{N_1}.
$$

By induction, assume for $1\le p\le l$ we have non-empty closed set $A_p^{1}$ with
$A_l^{1}\subset A_{l-1}^{1}\subset \cdots \subset A_{1}^{1} \subset R_{n,\pi}^\delta$,
and for all $m\in \N$ we have
$$
\overline{D}_{\mathcal{F}}(\{g\in G: (gx_m^1,\dots, gx_m^n)\in A_p^{1}\}) >\frac{\delta}{N_1N_2\cdots N_{p}}.
$$

For $p=l+1$,
since $A_l^{1}$ is a closed set,
$A_l^{1}$ can be expressed as a finite union of non-empty closed sets of diameter less than $\frac{1}{l+1}$,
i.e., $A_l^{1}= \cup_{i=1}^{N_{l+1}}A_{l+1}^i$ and $\diam(A_{l+1}^i)<\frac{1}{l+1}$.
Then for each $m\in \N$ there is $1\leq N_{l+1,m}\leq N_{l+1}$ such that
$$
\overline{D}_{\mathcal{F}}(\{g\in G: (gx_m^1,\dots, gx_m^n)\in A_{l+1}^{N_{l+1,m}} \}) > \frac{\delta}{N_1N_2\cdots N_{l+1}}.
$$
Without loss of generality we assume $N_{l+1,m}=1$ for all $m\in \N$ and so
$$
\overline{D}_{\mathcal{F}}(\{g\in G: (gx_m^1,\dots, gx_m^n)\in A_{l+1}^{1}  \}) > \frac{\delta}{N_1N_2\cdots N_{l+1}}.
$$

Since $A_p^1$ are non-empty closed sets with $A_{p+1}^1\subset A_p^1$ and $\lim_{p\to \infty}\diam(A_{p}^1)=0$,
 there is a unique point $(z_1,\dots,z_n)\in \bigcap_{p=1}^{\infty} A_p^{1}\subset R_{n,\pi}^\delta$.
We claim that $(z_1,\dots,z_n)\in MS_{n,\pi}^e(X, G)$.
In fact, let $V$ be an open neighborhood of $(z_1,\dots,z_n)$,
and there is $l\in \N$ such that $A_{l}^{1} \subset V$.
By the construction for any $m$, there are $x_m^1,\dots, x_m^n\in W_m$ such that
$$
\overline{D}_{\mathcal{F}}(\{g\in G: (gx_m^1,\dots, gx_m^n)\in A_{l}^{1}  \}) > \frac{\delta}{N_1N_2\cdots N_{l}}
$$
and so
$$
\overline{D}_{\mathcal{F}}(\{g\in G: (gx_m^1,\dots, gx_m^n)\in V \})
> \frac{\delta}{N_1N_2\cdots N_{l}}
$$
for all $m\in \N$.
For any non-empty open set $U\subset X$,  since $x$ is a transitive point, there is $h\in G$ such that $hx\in U$. We can choose $m_0\in\N$ such that
$hW_{m_0}\subset U$. This implies that $hx_{m_0}^1,\cdots, hx_{m_0}^n\in U$, then
$$\{g\in G: (g(hx_{m_0}^1),\dots, g(hx_{m_0}^n))\in V\}
=\{q\in G: (qx_{m_0}^1,\dots, qx_{m_0}^n)\in V\}h^{-1} $$

By Lemma \ref{20260620}
\begin{align*}
&\overline{D}_{\mathcal{F}}(\{g\in G: (g(hx_{m_0}^1),\dots, h(hx_{m_0}^n))\in V\} ) \\
&=\overline{D}_{\mathcal{F}}(\{q\in G: (qx_{m_0}^1,\dots, qx_{m_0}^n)\in V\}h^{-1} )\\
&=\overline{D}_{\mathcal{F}}(\{q\in G: (qx_{m_0}^1,\dots, qx_{m_0}^n)\in V\} )
 > \frac{\delta}{N_1N_2\cdots N_{l}}.
\end{align*}

So we have $(z_1,\dots,z_n)\in MS_{n,\pi}^e(X, G)$, completing the proof.

(2)$\Rightarrow$(1)
Let $(x_1,\dots,x_n)\in MS_{n,\pi}^{e}(X,G)$ and $0<\delta_1=\min_{1 \leq i \neq j \leq m} d(x_{i}, x_{j})$.
Then there exists an open neighborhood $\prod_{i=1}^n U_i$ of $(x_1,\dots,x_n)$
such that $\min_{1\le i\neq j\le n}d(U_i,U_j)=\frac{1}{2}\delta_1$.
Then
there exists $\delta_2>0$ such that for any non-empty open set $U$ there are $y_1,y_2,\dots,y_n\in U$ with $\pi(y_1)=\pi(y_2)=\cdots =\pi(y_n)$
and a subset $F$ of $\Z_+$ with $\overline{D}_{\mathcal{F}}(F)>\delta_2$ such that $g y_i \in U_i$ for all $i=1,2,\dots,n$ and $g\in G$.
Let $\delta=\min\{\frac{1}{2}\delta_1,\delta_2\}$.
Then $\min_{1\le i\neq j\le n}d(g y_i,g y_j)\ge\min_{1\le i\neq j\le n}d(U_i,U_j)> \delta$ for all $g\in G$
with $\overline{D}_{\mathcal{F}}(F)>\delta$.
By Proposition \ref{prop:eq4} $\pi$ is relative mean $n$-sensitive.

\end{proof}

Next we study the relation between relative entropy tuples and relative mean sensitive tuples.

\begin{lem}\label{lem:IT-properties}
Let $\pi:(X,G)\ra (Y,G)$ be a factor map between TDSs and $2\le n\in\N$.  We have the following characterizations.
\begin{enumerate}
\item\label{lem:IT-properties:1} \cite[Proposition 13.5]{DZ15}  $h_{\text{top}}(G| \pi)>0$ if and only if $E^e_{n,\pi}(X,G)\neq\emptyset$ for every $2\le n\in\N$.

\item\label{lem:IT-properties:2} \cite[Theorem 13.7]{DZ15} There is a $\nu\in M(X,G)$ such that $E_{n,\pi}^\nu(X,G)=E_{n,\pi}(X,G)$.

\end{enumerate}

\end{lem}

\begin{thm}\label{thm:minimal-positive-entropy}
Let $\pi:(X,G)\ra (Y,G)$ be a factor map between minimal TDSs and $2\le n\in\N$. Then
$$
E_{n,\pi}(X,G)\subset MS_{n,\pi}(X,G).
$$
\end{thm}
\begin{proof}
By Lemma \ref{lem:IT-properties} \eqref{lem:IT-properties:2} there is a $\mu\in M(X,G)$
such that $E_{n,\pi}^\mu(X,G)=E_{n,\pi}(X,G)$. Let $\mu=\int_{\Omega}\theta \textrm{d} m(\theta)$ be the ergodic decomposition of $\mu$. For each $\theta\in \Omega$,
by Theorem \ref{thm:entr-tuple--ms-tuple} $E_{n,\pi}^{\theta}(X,G)\subset MS_{n,\pi}^{\theta}(X,G)$.
Since $X$ is minimal, then every ergodic measure on $X$ has full support.
This implies that $MS_{n,\pi}^{\theta}(X,G)\subset MS_{n,\pi}(X,G)$  and so
$$
\cup_{\theta\in \Omega}E_{n,\pi}^{\theta}(X,G)\subset MS_{n,\pi}(X,G).
$$
By Theorem \ref{et-decom}
there is $\Omega'\subset \Omega$ with $m(\Omega')=1$ such that
$$
\overline{\cup_{\theta\in \Omega'}E_{n,\pi}^{\theta}(X,G)}\setminus \Delta_n(X)=E_{n,\pi}^\mu(X,G).
$$
Observe that $MS_{n,\pi}(X,G)\cup \Delta_n(X)$ is a closed  subset of $X^{(n)}$. Hence we have
$$
E_{n,\pi}(X,G)=E_{n,\pi}^\mu(X,G)=\overline{\cup_{\theta\in \Omega'}E_{n,\pi}^{\theta}(X,G)}\setminus\Delta_n(X)\subset MS_{n,\pi}(X,G),
$$
%and so $MS_n(X,T)=X^{(n)}\setminus \Delta_n^\prime(X)$,
completing the proof.
\end{proof}

\begin{cor}\label{cor:minimal-positive-entropy}
Let $\pi:(X,G)\ra (Y,G)$ be a factor map between minimal TDSs with $h_{\text{top}}(G| \pi)>0$.
Then $\pi$ is relative mean $n$-sensitive for every $2\le n\in\N$.
\end{cor}
\begin{proof}
Since $h_{\text{top}}(G| \pi)>0$,
by Lemma \ref{lem:IT-properties} \eqref{lem:IT-properties:1}  $E_{n,\pi}^e(X,G)\not=\emptyset$ for any $2\le n\in\N$.
By Theorem \ref{thm:minimal-positive-entropy} $MS_{n,\pi}^e(X,G)\not=\emptyset$.
By Proposition \ref{thm:ms=ms-tuple} $\pi$ is relative mean $n$-sensitive for every $2\le n\in\N$.

\end{proof}

\end{document}